\documentclass[12pt,a4paper]{article}%
\usepackage{amsmath, amsfonts, amsthm,color, latexsym, amssymb}
\usepackage{amscd}
\usepackage{amsmath}
\usepackage{amsfonts}
\usepackage{amssymb}
\usepackage{graphicx}%
\providecommand{\U}[1]{\protect\rule{.1in}{.1in}}
\newtheorem{theorem}{Theorem}[section]
\newtheorem{proposition}[theorem]{Proposition}

\newtheorem{example}[theorem]{Example}

\newtheorem{remark}[theorem]{Remark}

\newtheorem{lemma}[theorem]{Lemma}

\begin{document}

\title{\textsc{Dominated bilinear forms and 2-homogeneous polynomials}}
\author{Geraldo Botelho\thanks{Supported by CNPq Grant 306981/2008-4.}~, Daniel Pellegrino\thanks{Supported by INCT-Matem\'{a}tica, CNPq Grants 620108/2008-8 (Ed. Casadinho), 308084/2006-3 (Ed. Universal) and 471686/2008-5.}~ and Pilar Rueda\thanks{Supported by Ministerio de
Ciencia e Innovaci\'{o}n MTM2008-03211/MTM. \hfill\newline This
paper will appear in Publications of the Research Institute for
Mathematical Sciences.\hfill\newline2000 Mathematics Subject
Classification: 46G25, 46B20, 46B28.}}
\date{}
\maketitle

\begin{abstract}
The main goal of this note is to establish a connection between the cotype of the Banach space $X$ and the parameters $r$ for which every 2-homogeneous polynomial on $X$ is $r$-dominated. Let $\cot X$ be the infimum of the cotypes assumed by $X$ and $(\cot X)^*$ be its conjugate. The main result of this note asserts that if $\cot X > 2$, then for every $1 \leq r < (\cot X)^*$ there exists a non-$r$-dominated 2-homogeneous polynomial on $X$.
\end{abstract}


\section*{Introduction}
The notion of $p$-dominated multilinear mappings and homogeneous polynomials between Banach
spaces plays an important role in the nonlinear theory of absolutely
summing operators. It was introduced by Pietsch \cite{pietsch} and has been investigated by several
authors since then (see, e.g., \cite{jfa1, proc} and references therein).

Let $X$ be a Banach space and $m$ be a positive integer. A continuous $m$-linear form $A$  on $X^m$  is $r$-dominated if  $(A(x_{j}^1, \ldots, x_j^m))_{j=1}^{\infty} \in \ell_{\frac{r}{m}}$  whenever $(x_{j}^1)_{j=1}^{\infty},  \ldots, (x_{j}^m)_{j=1}^{\infty}$ are  weakly $r$-summable in $X$. In a similar way, a scalar-valued $m$-homogeneous polynomial $P$ on $X$ is $r$-dominated if $(P(x_{j}))_{j=1}^{\infty} \in \ell_{\frac{r}{m}}$ whenever  $(x_{j})_{j=1}^{\infty}$ is weakly $r$-summable in $X$.

In \cite[Lemma 5.4]{Jar} it is proved that for every infinite-dimensional Banach space $X$, every $p \geq 1$ and every $m \geq 3$, there exists a continuous non-$p$-dominated $m$-linear form on $X^m$. For polynomials the question has been recently settled in \cite{proc}, where it is proved that for every infinite-dimensional Banach space $X$, every $p \geq 1$ and every $m \geq 3$, there exists a continuous non-$p$-dominated scalar-valued $m$-homogeneous polynomial on $X$. So, coincidence situations can occur only for $m = 2$. Sometimes it happens that every continuous bilinear form on $X^2$ is 2-dominated, for example if $X$ is either an ${\cal L}_\infty$-space, the disc algebra $\cal A$ or the Hardy space $H^\infty$ (see \cite[Proposition 2.1]{PAMS}). In this case every continuous bilinear form on $X^2$ and every continuous scalar-valued 2-homogeneous polynomial on $X$ are $r$-dominated for every $r \geq 2$. But what about $r$-dominated bilinear forms and 2-homogeneous polynomials for $1 \leq r < 2$? \\
\indent Those spaces $X$ that enjoy the property that all bilinear forms on $X^2$ are 1-dominated are all of cotype 2 (Example \ref{ex}). In Proposition \ref{propos} we see that having cotype $2 + \varepsilon$ for every $\varepsilon > 0$ is a necessary condition. So, for a space $X$ such that $\cot X > 2$ it is natural to investigate how close $r$ can be to 1 with the property that every bilinear form on $X^2$ (or 2-homogeneous polynomial on $X$) is $r$-dominated. For bilinear forms it is not difficult to see (Proposition \ref{proposition}) that $(\cot X)^*$, the conjugate of the number $\cot X$, is the closest $r$ can be to 1. As usual, for polynomials the situation is more delicate. In the main result of this paper, Theorem \ref{coro11}, we prove that the estimate $(\cot X)^*$ holds for 2-homogeneous polynomials as well. We also point out that this result is somehow sharp.

\section{Notation}

Throughout this paper $n$ and $m$ are positive integers, $X$ and $Y$ will
stand for Banach spaces over $\mathbb{K}=\mathbb{R}$ or $\mathbb{C}$. The
Banach spaces of all continuous $m$-linear mappings $A \colon X^m \longrightarrow Y$ and continuous $m$-homogeneous polynomials $P\colon
X\longrightarrow Y,$ endowed with the usual $\sup$ norms, are denoted by $\mathcal{L}(^{m}X;Y)$ and $\mathcal{P}(^{m}X;Y)$, respectively ($\mathcal{L}(X;Y)$ if $m=1$). When $m=1$ and $Y=\mathbb{K}$ we
write $X^{\ast}$ to denote the topological dual of $X$. The closed unit
ball of $X$ is represented by $B_{X}.$ The notation $\cot X$ denotes the
infimum of the cotypes assumed by $X$. The identity operator on $X$ is denoted
by $id_{X}$. For details on the theory of multilinear mappings and homogeneous polynomials between Banach spaces we
refer to \cite{Di, Mu}.

Given $r\in\lbrack0,\infty)$, let $\ell_{r}(X)$ be the Banach ($r$-Banach if
$0 < r < 1$) space of all absolutely $r$-summable sequences $(x_{j}%
)_{j=1}^{\infty}$ in $X$ with the norm $\Vert(x_{j})_{j=1}^{\infty}\Vert
_{r}=(\sum_{j=1}^{\infty}\Vert x_{j}\Vert^{r})^{1/r}$. We denote by $\ell
_{r}^{w}(X)$ the Banach ($r$-Banach if $0 < r < 1$) space of all weakly
$r$-summable sequences $(x_{j})_{j=1}^{\infty}$ in $X$ with the norm
$\Vert(x_{j})_{j=1}^{\infty}\Vert_{w,\,r}=\sup_{\varphi\,\in\,B_{X^{\ast}}%
}\Vert(\varphi(x_{j}))_{j=1}^{\infty}\Vert_{r}$.

Let $p,q>0.$ An $m$-linear mapping $A\in\mathcal{L}(^{m}X;Y)$ is
absolutely $(p;q)$-summing if $\left(  A(x_{j}^1, \ldots, x_j^m)\right)  _{j=1}^{\infty}\in
\ell_{p}(Y)$ whenever $\left(  x_{j}^1\right)  _{j=1}^{\infty},\ldots, \left(  x_{j}^m\right)  _{j=1}^{\infty}\in\ell_{q}%
^{w}(X).$ It is well-known that $A$ is absolutely $(p;q)$-summing if and only
if there is a constant $C\geq0$ so that%
\[
\left(
{\sum\limits_{j=1}^{n}}
\left\Vert A(x_{j}^1, \ldots, x_j^m)\right\Vert ^{p}\right)  ^{\frac{1}{p}}\leq C \prod_{k=1}^m
\Vert(x_{j}^k)_{j=1}^{n}\Vert_{w,\,q}%
\]
for every $x_{1}^k,\ldots,x_{n}^k\in X$, $k = 1, \ldots, m$, and $n$ positive integer. The infimum of
such $C$ is denoted by $\Vert A\Vert_{as(p;q)}.$ The space of all
absolutely $(p;q)$-summing $m$-linear mappings from $X^m$ to $Y$ is
denoted by $\mathcal{L}_{as(p;q)}(^{m}X;Y)$ and $\Vert\cdot\Vert_{as(p;q)}$ is
a complete norm ($p$-norm if $p<1$) on $\mathcal{L}_{as(p;q)}(^{m}X;Y)$.\\
\indent An $m$-homogeneous polynomial $P\in\mathcal{P}(^{m}X;Y)$ is
absolutely $(p;q)$-summing if the symmetric $m$-linear mapping associated to $P$ is absolutely $(p;q)$-summing, or, equivalently, if $\left(  P(x_{j})\right)  _{j=1}^{\infty}\in
\ell_{p}(Y)$ whenever $\left(  x_{j}\right)  _{j=1}^{\infty}\in\ell_{q}%
^{w}(X).$ It is well-known that $P$ is absolutely $(p;q)$-summing if and only
if there is a constant $C\geq0$ so that%
\[
\left(
{\sum\limits_{j=1}^{n}}
\left\Vert P(x_{j})\right\Vert ^{p}\right)  ^{\frac{1}{p}}\leq C\left(
\Vert(x_{j})_{j=1}^{n}\Vert_{w,\,q}\right)  ^{m}%
\]
for every $x_{1},\ldots,x_{n}\in X$ and $n$ positive integer. The infimum of
such $C$ is denoted by $\Vert P\Vert_{as(p;q)}.$ The space of all
absolutely $(p;q)$-summing $m$-homogeneous polynomials from $X$ to $Y$ is
denoted by $\mathcal{P}_{as(p;q)}(^{m}X;Y)$ and $\Vert\cdot\Vert_{as(p;q)}$ is
a complete norm ($p$-norm if $p<1$) on $\mathcal{P}_{as(p;q)}(^{m}X;Y)$.

An $m$-homogeneous polynomial $P\in\mathcal{P}(^{m}X;Y)$ is said to
be $r$\textit{-dominated} if it is absolutely
$(\frac{r}{m};r)$-summing. In this case we write
$\mathcal{P}_{d,r}(^{m}X;Y)$ and $\Vert\cdot\Vert_{d,r}$ instead of
$\mathcal{P}_{as(\frac{r}{m};r)}(^{m}X;Y)$ and $\Vert\cdot\Vert
_{as(\frac{r}{m};r)}$. As usual  we write $\mathcal{P}_{d,r}(^{m}X)$
and $\mathcal{P}(^{m}X)$ when $Y=\mathbb{K}$. The definition (and
notation) for $r$-dominated multilinear mappings is analogous (for
the notation just replace $\cal P$ by $\cal L$). For details we refer to \cite{note, PAMS, Jar}.

\section{Results}\label{two}
First we establish the existence of Banach spaces on which every bilinear form (hence every scalar-valued 2-homogeneous polynomial) is 1-dominated. By $X \tilde{\otimes}_\pi X$ and $X \tilde{\otimes}_\varepsilon X$ we mean the completions of the tensor product $ X \otimes X$ with respect to the projective norm $\pi$ and the injective norm $\varepsilon$, respectively. For the basics on tensor norms we refer to \cite{DF, ryan}.\\
\indent By $\Pi_r$ we denote the ideal of absolutely $r$-summing linear operators. The following well-known factorization theorem (see, e.g., \cite[Theorem 14]{pietsch} or \cite[Proposition 46(a)]{note}) will be useful a couple of times.

\begin{lemma}\label{lemma} ${\cal L}_{d,r}(^mX;Y) = {\cal L} \circ (\Pi_r,\stackrel{(m)}{\ldots},\Pi_r)(^mX;Y)$ and ${\cal P}_{d,r}(^mX;Y) = {\cal P} \circ \Pi_r(^mX;Y)$ regardless of the positive integer $m$ and the Banach spaces $X$ and $Y$.
\end{lemma}

\begin{proposition}\label{prop} Let $X$ be a cotype 2 space. Then $X \tilde{\otimes}_\pi X = X \tilde{\otimes}_\varepsilon X$ if and only if $\mathcal{L}_{d,1}(^{2}X)= \mathcal{L}(^{2}X)$.
\end{proposition}

\begin{proof} Arguably this result is contained, in essence, in \cite{Jar}. We give the details for the sake of completeness.\\
\indent Assume that $X \tilde{\otimes}_\pi X = X \tilde{\otimes}_\varepsilon X$ and let $A \in \mathcal{L}(^{2}X)$. Denoting the linearization of $A$ by $A_L$ we have that $A_L \in (X \tilde{\otimes}_\pi X)' = (X \tilde{\otimes}_\varepsilon X)'$. Regarding $X$ as a subspace of $C(B_{X'})$ and using that $\varepsilon$ respects the formation of subspaces, $A_L$ admits a continuous extension to $C(B_{X'})\tilde{\otimes}_\varepsilon C(B_{X'}) $, hence to $C(B_{X'})\tilde{\otimes}_\pi C(B_{X'}) $ because $\varepsilon \leq \pi$. Using that bilinear forms on $C(K)$-spaces are 2-dominated, we have that the bilinear form associated to this extension is 2-dominated. But restrictions of 2-dominated bilinear forms are 2-dominated as well, so $A$ is 2-dominated. 
Since 2-summing operators on cotype 2 spaces are 1-summing \cite[Corollary 11.16(a)]{DJT}, we have that $\Pi_1(X;Y) = \Pi_2(X;Y)$ for every $Y$, so by Lemma \ref{lemma} we have
$${\cal L}_{d,2}(^2X) = {\cal L} \circ (\Pi_2,\Pi_2)(^2X) = {\cal L} \circ (\Pi_1,\Pi_1)(^2X) = {\cal L}_{d,1}(^2X).$$
It follows that $A$ is 1-dominated.\\
\indent Conversely, assume that $\mathcal{L}_{d,1}(^{2}X)= \mathcal{L}(^{2}X)$ and let $A \in  \mathcal{L}(^{2}X)$. Since 1-dominated bilinear forms are 2-dominated, we have that $A$ is 2-dominated, hence extendible by \cite[Theorem 23]{MT}. Adapting the proof of \cite[Proposition 1.1]{carando} to bilinear forms we conclude that $A$ is integral. Now apply \cite[Ex. 4.12]{DF} to get $X \tilde{\otimes}_\pi X = X \tilde{\otimes}_\varepsilon X$.
\end{proof}

\begin{example}\label{ex}\rm Pisier \cite{acta} proved that every cotype 2 space $E$ embeds isometrically in a cotype 2 space $X$ such that $X \tilde{\otimes}_\pi X = X \tilde{\otimes}_\varepsilon X$. So for every such space $X$ we have that $\mathcal{L}_{d,1}(^{2}X)= \mathcal{L}(^{2}X)$.
\end{example}

It is easy to see that $\cot X = 2$ is a necessary condition for every bilinear form on $X$ to be 1-dominated:

\begin{proposition}\label{propos} If $\mathcal{L}_{d,1}(^{2}X)= \mathcal{L}(^{2}X)$, then $\cot X = 2$.
\end{proposition}

\begin{proof} By \cite[Lemma 3.4]{B} we have that every bounded linear operator from $X$ to $X'$ is 1-summing. So, from \cite[Proposition 8.1(2)]{LP} we conclude that the identity operator on $X$ is (2;1)-summing. It follows that $\cot X = 2$ by \cite[Theorem 14.5]{DJT}.
\end{proof}

Let $X$ be such that $\cot X > 2$. Since we cannot have $\mathcal{L}_{d,1}(^{2}X)= \mathcal{L}(^{2}X)$, for which $r > 1$ is it possible to have $\mathcal{L}_{d,r}(^{2}X)= \mathcal{L}(^{2}X)$? Or, at least, $\mathcal{P}_{d,r}(^{2}X)= \mathcal{P}(^{2}X)$? In other words we seek estimates for the numbers
$${\cal L}_X := \inf\{r : \mathcal{L}_{d,r}(^{2}X)= \mathcal{L}(^{2}X)\}~{\rm and~}{\cal P}_X := \inf\{r : \mathcal{P}_{d,r}(^{2}X)= \mathcal{P}(^{2}X)  \}. $$
It is not difficult to give a lower bound for ${\cal L}_X$. By $q^*$ we mean the conjugate of the number $q > 1$.

\begin{proposition}\label{proposition} If $\cot X > 2$, then ${\cal L}_X \geq (\cot X)^*$.
\end{proposition}

\begin{proof} By Proposition \ref{propos} we know that $\mathcal{L}_{d,1}(^{2}X)\neq \mathcal{L}(^{2}X)$. Using that $\Pi_1(X;Y) = \Pi_r(X;Y)$ whenever $1 \leq r < (\cot X)^*$ \cite[Corollary 11.16(b)]{DJT} and Lemma \ref{lemma}, we have that
$$\mathcal{L}_{d,r}(^{2}X) = {\cal L} \circ (\Pi_r,\Pi_r)(^2X) = {\cal L} \circ (\Pi_1,\Pi_1)(^2X) = \mathcal{L}_{d,1}(^{2}X) \neq \mathcal{L}(^{2}X)$$
for every $1 \leq r < (\cot X)^*$, so the result follows.
\end{proof}

It is in principle not clear that the same holds for polynomials. For polynomials the situation is usually more delicate. For instance, in \cite{proc} one can find a non-$r$-dominated bilinear form whose associated 2-homogeneous polynomial happens to be $r$-dominated. Though not clear in principle, we shall prove that ${\cal P}_X \geq (\cot X)^*$.

\indent The following proof extends an argument which
was first used in this context in \cite{studia}.

\begin{theorem}
\label{con}Let $m$ be an even positive integer and $X$ be an infinite
dimensional real Banach space. If $q<1$ and $\mathcal{P}_{as(q;r)}%
(^{m}X)=\mathcal{P}(^{m}X),$ then $id_{X}$ is $(\frac{mq}{1-q},r)$-summing.
\end{theorem}

\begin{proof}
The open mapping theorem gives us a constant $K>0$ so that $\left\Vert
Q\right\Vert _{as(q;r)}\leq K\Vert Q\Vert$ for all continuous $m$-homogeneous
polynomials $Q\colon X\longrightarrow Y.$\newline\indent Let $n\in\mathbb{N}$
and $x_{1},\ldots,x_{n}\in X$ be given. Consider $x_{1}^{\ast},\ldots
,x_{n}^{\ast}\in B_{X^{\ast}}$ so that $x_{j}^{\ast}(x_{j})=\left\Vert
x_{j}\right\Vert $ for every $j=1,\ldots,n$. Let $\mu
_{1},\ldots,\mu_{n}$ be real numbers such that $\sum\limits_{j=1}^{n}|\mu
_{j}|^{s}=1,$ where $s=\frac{1}{q}.$ Define $P\colon X\longrightarrow
\mathbb{R}$ by
\[
P(x)=\sum\limits_{j=1}^{n}\left\vert \mu_{j}\right\vert ^{\frac{1}{q}}%
x_{j}^{\ast}(x)^{m}\mathrm{~for~every~}x\in X.
\]
Since $m$ is even and $\mathbb{K} = \mathbb{R}$, it follows that
$P(x)\geq0$ for every $x\in X$. Also, $ \left\vert P(x)\right\vert
=P(x)\geq\left\vert \mu_{k}\right\vert ^{\frac
{1}{q}}x_{k}^{\ast}(x)^{m} $ for every $x\in X$ and every
$k=1,\ldots,n.$ From
\[
\left\vert P(x)\right\vert =\left\vert \sum\limits_{j=1}^{n}\left\vert \mu
_{j}\right\vert ^{\frac{1}{q}}x_{j}^{\ast}(x)^{m}\right\vert \leq\Vert
x\Vert^{m}\sum\limits_{j=1}^{n}|\mu_{j}|^{\frac{1}{q}}=\Vert x\Vert^{m}%
\]
we conclude that $\left\Vert P\right\Vert _{as(q;r)}\leq K\Vert P\Vert\leq K$.
So
\begin{align*}
\left(  \sum\limits_{j=1}^{n}\Vert x_{j}\Vert^{mq}\left\vert \mu
_{j}\right\vert \right)  ^{\frac{1}{q}}  &  =\left(  \sum\limits_{j=1}%
^{n}\left(  \Vert x_{j}\Vert^{m}\left\vert \mu_{j}\right\vert ^{\frac{1}{q}%
}\right)  ^{q}\right)  ^{\frac{1}{q}}\leq\left(  \sum\limits_{j=1}%
^{n}\left\vert P(x_{j})\right\vert ^{q}\right)  ^{\frac{1}{q}}\\
&  \leq\left\Vert P\right\Vert _{as(q;r)}(\Vert(x_{j})_{j=1}^{n}\Vert
_{w,r})^{m}.
\end{align*}
Observing that this last inequality holds whenever $\sum\limits_{j=1}^{n}%
|\mu_{j}|^{s}=1$ and that $\frac{1}{s}+\frac{1}{\frac{s}{s-1}}=1$ we have
\begin{align*}
\left(  \sum\limits_{j=1}^{n}\Vert x_{j}\Vert^{\frac{s}{s-1}mq}\right)
^{\frac{1}{\frac{s}{s-1}}}  &  =\sup\left\{  \left\vert \sum\limits_{j=1}%
^{n}\mu_{j}\Vert x_{j}\Vert^{mq}\right\vert ;\sum\limits_{j=1}^{n}|\mu
_{j}|^{s}=1\right\} \\
&  \leq\sup\left\{  \sum\limits_{j=1}^{n}\left\vert \mu_{j}\right\vert \Vert
x_{j}\Vert^{mq};\sum\limits_{j=1}^{n}|\mu_{j}|^{s}=1\right\} \\
&  \leq\left\Vert P\right\Vert _{as(q;r)}^{q}(\Vert(x_{j})_{j=1}^{n}%
\Vert_{w,r})^{mq}\\
&  \leq K^{q}(\Vert(x_{j})_{j=1}^{n}\Vert_{w,r})^{mq}.
\end{align*}
It follows that%
\[
\left(  \sum\limits_{j=1}^{n}\Vert x_{j}\Vert^{\frac{s}{s-1}mq}\right)
^{\frac{1}{(\frac{s}{s-1})mq}}\leq K^{\frac{1}{m}}\cdot\Vert(x_{j})_{j=1}%
^{n}\Vert_{w,r}.
\]
Since $\frac{s}{s-1}mq=\frac{mq}{1-q}$, $n$ and $x_{1},\ldots,x_{n}\in X$ are
arbitrary, we conclude that $id_{X}$ is $(\frac{mq}{1-q},r)$-summing.
\end{proof}

The following theorem holds for spaces over $\mathbb{K} = \mathbb{R}$ or $\mathbb{C}$:

\begin{theorem}
\label{coro11}If $\cot X=q>2$, then $\mathcal{P}_{d,r}(^{2}X)\neq
\mathcal{P}(^{2}X)$ for $1\leq r<q^{\ast}$, where $q^{\ast}$ is the
conjugate of $q$. In other words, ${\cal P}_X \geq q^*$.
\end{theorem}

\begin{proof} Real case: Let $1\leq r<q^{\ast}$. Combining Lemma \ref{lemma} and \cite[Corollary 11.16(b)]{DJT}
it is immediate that $ \mathcal{P}_{d,r}(^{2}X)=\mathcal{P}_{d,1}(^{2}X) .$ If
$\mathcal{P}_{d,r}(^{2}X)= \mathcal{P}_{d,1}(^{2}X) =
\mathcal{P}(^{2}X),$ from Theorem \ref{con} we could conclude
that $id_{X}$ is $(2;1)$-summing, but this is impossible because
$\cot X>2.$ \\
Complex case: If $X$ is a complex Banach space, $\cot X=q>2$ and $1\leq r<q^{\ast}$, by \cite[Lemma 3.1]{bbjp} we know that $\cot X_{\mathbb{R}}=q>2$, so there is a
non-$r$-dominated polynomial $P \in {\cal P}(^2X_\mathbb{R})$. Denoting by $\widetilde{P}$ the
complexification of $P$ we have that $\widetilde{P} \in {\cal P}(^2X)$ and following the lines of
\cite[Proposition 4.30]{David} it is not difficult to prove that $\widetilde{P}$ fails to be $r$-dominated either.
\end {proof}

\begin{remark}\rm
\label{hhggff} Let $X$ be any of the spaces constructed by Pisier \cite{acta}. By Example \ref{ex} we know that $\mathcal{P}_{d,1}(^{2}X)= \mathcal{P}(^{2}X)$, which makes clear that  Theorem \ref{coro11} is sharp in the
sense that it is not valid for cotype $2$ spaces.
\end{remark}

\noindent {\bf Conjecture.} We conjecture that if $X$ is infinite-dimensional and $\mathcal{L}_{d,1}(^{2}X)= \mathcal{L}(^{2}X)$, then $X \tilde{\otimes}_\pi X = X \tilde{\otimes}_\varepsilon X$. Observe that for an infinite-dimensional space $X$ with $\mathcal{L}_{d,1}(^{2}X)= \mathcal{L}(^{2}X)$ and $X \tilde{\otimes}_\pi X \neq X \tilde{\otimes}_\varepsilon X$, if any, we should have:\\
$\bullet$ $X$ has no unconditional basis \cite[Theorem 3.2]{PAMS};\\
$\bullet$ $X$ has cotype $2 + \varepsilon$ for every $\varepsilon > 0$ (Proposition \ref{propos});\\
$\bullet$ $X$ does not have cotype 2 (Proposition \ref{prop});\\
$\bullet$ $X'$ is a GT space \cite[Theorem 3.4]{Jar};\\
$\bullet$ Every linear operator from $X$ to $X'$ is absolutely 1-summing (by \cite[Lemma 3.4]{B} this is a consequence of  $\mathcal{L}_{d,1}(^{2}X)= \mathcal{L}(^{2}X)$), in particular $X$ is Arens-regular;\\
$\bullet$ Not every linear operator from $X$ to $X'$ is integral (this is a consequence of $X \tilde{\otimes}_\pi X \neq X \tilde{\otimes}_\varepsilon X$).

\vspace*{1em} \noindent[Geraldo Botelho] Faculdade de Matem\'atica,
Universidade Federal de Uberl\^andia, 38.400-902 - Uberl\^andia, Brazil,
e-mail: botelho@ufu.br.

\medskip

\noindent[Daniel Pellegrino] Departamento de Matem\'atica, Universidade
Federal da Para\'iba, 58.051-900 - Jo\~ao Pessoa, Brazil, e-mail: dmpellegrino@gmail.com.

\medskip

\noindent[Pilar Rueda] Departamento de An\'alisis Matem\'atico, Universidad de
Valencia, 46.100 Burjasot - Valencia, Spain, e-mail: pilar.rueda@uv.es.

\end{document}